\documentclass[12pt,a4paper]{article}
\usepackage{amssymb,amsmath,amsthm}
\usepackage{graphicx}
\usepackage[colorlinks=true, allcolors=blue]{hyperref}

\newcommand\cF{{\mathcal F}}

\newcommand\cG{{\mathcal G}}
\newcommand\cH{{\mathcal H}}

\theoremstyle{plain}
\newtheorem{theorem}{Theorem}[section]
\newtheorem{lemma}[theorem]{Lemma}

\theoremstyle{definition}

\newtheorem{claim}[theorem]{Claim}

\newcommand\cref[1]{Corollary~\ref{cor:#1}}

\title{On $L$-close Sperner systems}
\author{D\'aniel T. Nagy$^1$ \hskip 1truecm Bal\'azs Patk\'os$^{1,2}$ 
\medskip \\
\small $^1$ Alfr\'ed R\'enyi Institute of Mathematics, P.O.B. 127, Budapest H-1364, Hungary.\\
\small  
$^2$Lab. of Combinatorial and Geometric Structures, Moscow Inst. of Physics and Technology\\
\medskip
\small \texttt{\{nagydani,patkos\}@renyi.hu}}
\date{}

\begin{document}

\maketitle

\begin{abstract}
    For a set $L$ of positive integers, a set system $\cF \subseteq 2^{[n]}$ is said to be $L$-close Sperner, if for any pair $F,G$ of distinct sets in $\cF$ the skew distance $sd(F,G)=\min\{|F\setminus G|,|G\setminus F|\}$ belongs to $L$. We reprove  an extremal result of Boros, Gurvich, and Milani\v c on the maximum size of $L$-close Sperner set systems for $L=\{1\}$ and generalize to $|L|=1$ and obtain slightly weaker bounds for arbitrary $L$. We also consider the problem when $L$ might include 0 and reprove a theorem of Frankl, F\"uredi, and Pach on the size of largest set systems with all skew distances belonging to $L=\{0,1\}$.
\end{abstract}
\section{Introduction}

One of the first results of extremal finite set theory is Sperner's theorem \cite{S} that states that if for any pair $F,F'$ of distinct sets in a set systems $\cF\subseteq 2^{[n]}$ we have $\min\{|F\setminus F'|,|F'\setminus F|\}\ge 1$, then $|\cF| \le \binom{n}{\lfloor n/2\rfloor|}$ holds. Set systems with this property are called \textit{antichains} or \textit{Sperner systems}. This theorem has lots of generalizations and applications in different areas of mathematics (see the book \cite{E} and Chapter 3 of \cite{GP}). Recently, Boros, Gurvich, and Milani\v c introduced the following notion:
given a positive integer $k$,  we say that a set system $\cF$ is \textit{$k$-close Sperner} if every pair $F,G\in \cF$ of distinct sets satisfies
$1\le \min\{|F\setminus G||,|G\setminus F|\}\le k$. In particular, $\cF$ is 1-close Sperner if every pair $F,G\in \cF$ of distinct sets satisfies
$\min\{|F\setminus G||,|G\setminus F|\}=1$. (The authors used the unfortunate \textit{$k$-Sperner} term which, throughout the literature, refers to set systems that are union of $k$ many antichains. That is why we decided to use instead the terminology $k$-close Sperner systems.) Boros, Gurvich, and Milani\v c's motivation to study these set systems comes from computer science: they wanted to compare them to other classes of Sperner systems (see also \cite{BGMa} and \cite{CM}). They obtained some structural results from which they deduced the following extremal theorem. For a set $F\subseteq [n]=\{1,2,\dots,n\}$, its \textit{characteristic vector} $v_F$ is a 0-1 vector of length $n$ with $(v_F)_i=1$ if and only if $i\in F$.

\begin{theorem}[Boros, Gurvich, Milani\v c \cite{BGMb}]\label{regi}
If the set system $\{\emptyset\}\neq\{F_1,F_2\dots,F_m\}\subseteq 2^{[n]}$ is 1-close Sperner, then the characteristic vectors $v_{F_1},v_{F_2},\dots,v_{F_m}$ are linearly independent over $\mathbb{R}$. In particular, $m\le n$.
\end{theorem}

In this short note, we reprove the extremal part of Theorem \ref{regi} via a different linear algebraic approach and generalize the result. For a subset $L$ of $[n]$,  we say that a set system $\cF$ is \textit{$L$-close Sperner} if every pair $F,G\in \cF$ satisfies
$\min\{|F\setminus G|,|G\setminus F|\}\in L$. Our first result is the following.

\begin{theorem}\label{main}
If the set system $\{F_1,F_2\dots,F_m\}\subseteq 2^{[n]}$ is $L$-close Sperner for some $L\subseteq [n]$, then we have $m\le \sum_{h=0}^{|L|}\binom{n}{h}$. Furthermore, if $|L|=1$, then $m\le n$ holds.
\end{theorem}

Note that if $|L|$ is fixed and $n$ tends to infinity, then the bound is asymptotically sharp as shown by $L=\{1,2,\dots,k\}$ (i.e. the $k$-close Sperner property) and the set system $\binom{[n]}{k}=\{F\subseteq [n]:|F|=k\}$. Observe also that the inequality $m\le n$ is sharp for $L=\{1\}$ as shown by the family of singletons, but there exist many other $1$-close Sperner systems with $n$ sets. Furthermore, if $L=\{q\}$ for some prime power $q$ and $n=q^2+q+1$, then the lines of a projective plane of order $q$ form an $L$-close family of size $n$, so the bound $m\le n$ is sharp in this case, too.

Apart from Sperner-type theorems, the other much studied area in extremal finite set theory are intersection properties (see e.g. Chapter 2 of \cite{GP}). For a set $L$ of integers, a set system $\cF$ is said to be $L$-intersecting if for any pair $F,F'$ of distinct sets in $\cF$ we have $|F\cap F'|\in L$. Frankl and Wilson \cite{FW} proved the same upper bound $\sum_{h=0}^{|L|}\binom{n}{h}$ on the size of $L$-intersecting set systems. Frankl and Wilson used higher incidence matrices to prove their result, but later the polynomial method (see \cite{BF} and \cite{ABS}) turned out to be very effective in obtaining $L$-intersection theorems. In the proof of the moreover part of Theorem \ref{main}, an additional idea due to Blokhuis \cite{B} will be used.

We will need the following well-known lemma, we include the proof for sake of completeness. For any field $\mathbb{F}$, we denote by $\mathbb{F}^n[x]$ the vector space over $\mathbb{F}$ of polynomials of $n$ variables with coefficients from $\mathbb{F}$.

\begin{lemma}
\label{indep_poly} Let $p_1(x),p_2(x),\dots,p_m(x)\in \mathbb{F}^n[x]$ be polynomials and $v_1,v_2,\dots, v_m\in \mathbb{F}^n$ be vectors such that $p_i(v_i)\neq 0$ and $p_i(v_j)=0$ holds for all $1\le j<i\le m$. Then the polynomials are linearly independent.
\end{lemma}

\begin{proof}
Suppose that $\sum_{i=1}^mc_ip_i(x)=0$. As $p_i(v_1)=0$ for all $1<i$ we obtain $c_1p_1(v_1)=0$ and therefore $c_1=0$ holds. We proceed by induction on $j$. If $c_h=0$ holds for all $h<j$, then using this and $p_i(v_j)=0$ for all $i>j$, we obtain $c_jp_j(v_j)= 0$ and therefore $c_j=0$.
\end{proof}

Results on $L$-intersecting families had some geometric consequences on point sets in $\mathbb{R}^n$ defining only a few distances, in particular on set systems $\cF$ with only a few Hamming distance. The \textit{skew distance} $sd(F,G):=\min\{|F\setminus G|,|G\setminus F|\}$ does not define a metric space on $2^{[n]}$ as $sd(F,G)=0$ holds if and only if $F\subseteq G$ or $G\subseteq F$ and one can easily find triples for which the triangle inequality fails: if $A$ is the set of even integers in $[n]$, $C$ is the set of odd integers in $[n]$, and $B=\{1,2\}$, then $\lfloor n/2\rfloor=sd(A,C)\not\le sd(A,B)+sd(B,C)=1+1$

One can also investigate the case when $L$ includes 0. Then set systems with the required property are not necessarily Sperner, so we will say that $\cF$ is $L-$skew distance (or $L$-sd for short) if $sd(A,B)\in L$ for all pairs of distinct sets $A,B\in \cF$. We will write $ex_{sd}(n,L)$ to denote the largest size of an $L$-skew distance system $\cF\subseteq 2^{[n]}$. Observe that $ex_{sd}(n,\{0\})$ asks for the maximum size of a chain in $2^{[n]}$ which is obviously $n+1$. This shows that the moreover part of Theorem \ref{main} does not remain valid in this case. In a different context Frankl, F\"uredi, and Pach considered the case $L=\{0,1,\dots,t\}$. They considered the following construction: let $\emptyset=C_0\subsetneq C_1 \subsetneq C_2 \subsetneq \dots \subsetneq C_n=[n]$ be a maximal chain and let
$$\cF_{n,t}=\{F: C_{|F|-t}\subset F\}\cup \{F: |F|\le t ~\text{or}\ |F|\ge n-t\}.$$
The size of $\cF_{n,t}$ is $\binom{n}{t+1}-\binom{2t+1}{t+1}+2\sum_{i=0}^t\binom{n}{i}$ and clearly $\cF_{n,t}$ is $\{0,1,\dots,t\}$-sd. This gives the lower bounds in the following results.

\begin{theorem}[Frankl, F\"uredi, Pach, \cite{FFP}]\label{induction}
If $n\ge 3$, we have $ex_{sd}(n,\{0,1\})=\binom{n}{2}+2n-1$.
\end{theorem}

\begin{theorem}[Frankl, F\"uredi, Pach, \cite{FFP}]\label{ffp}
For any $n,t$ with $n\ge 2(t+2)$, we have \\ $\binom{n}{t+1}-\binom{2t+1}{t+1}+2\sum_{i=0}^t\binom{n}{i} \le ex_{sd}(n,\{0,1,\dots,t\})<\binom{n}{t+1}+5(t+1)^2\binom{n}{t}$.
\end{theorem}

The authors of \cite{FFP} conjectured that the lower bound is tight in Theorem \ref{ffp} for large enough $n$. (There are larger constructions for small $n$.) We will give a simple, new proof of Theorem \ref{induction} that proceeds by induction.
 
\section{Proof and remarks}
We start by introducing some notation. For two vectors, $u,v$ of length $n$ we denote their scalar product $\sum_{i=1}^nu_iv_i$ by $u\cdot v$. We will often use the fact that for any pair $F,G$ of sets we have $v_F\cdot v_G=|F\cap G|$. We will also use that $\min\{|F\setminus G|,|G\setminus F|\}=|F\setminus G|$ if and only if $|F|\le |G|$ holds.

For two sets $F,L\subseteq [n]$ we define the polynomial $p'_{F,L}\in\mathbb{R}^n[x]$ as 
$$p'_{F,L}(x)=\prod_{h\in L}(|F|-v_F\cdot x-h).$$
We obtain $p_{F,L}(x)$ from $p'_{F,L}(x)$ by replacing every $x_i^t$ term by $x_i$ for every $t\ge 2$ and $i=1,2,\dots,n$. As $0=0^t$ and $1=1^t$ for any $t\ge 2$, we have $p_{F,L}(v_G)=p'_{F,L}(v_G)=\prod_{h\in L}(|F\setminus G|-h)$. Finally, observe that the polynomials $p_{F,L}(x)$ all belong to the subspace $M_{|L |}$ of $\mathbb{R}^n[x]$ spanned by $\{x_{i_1}x_{i_2}\dots x_{i_l}:0\le l\le |L|,i_1<i_2<\dots<i_l\}$, where $l=0$ refers to the constant 1 polynomial $\mathbf{1}$. Note that $\dim(M_{|L|})=\sum_{i=0}^{|L|}\binom{n}{i}$.

Based on the above, Theorem \ref{main} is an immediate consequence of the next result.

\begin{theorem}
If the set system $\{F_1,F_2\dots,F_m\}\subseteq 2^{[n]}$ is $L$-close Sperner, then the polynomials $p_{F_1,L}(x),p_{F_2,L}(x),\dots,p_{F_m,L}(x)$ are linearly independent in $\mathbb{R}^n[x]$. In particular, $m\le \sum_{h=0}^{|L|}\binom{n}{h}$. Moreover, if $|L|=1$ and $\{F_1,F_2\dots,F_m\}\neq \{\emptyset\}$, then the polynomials $p_{F_1,L}(x),p_{F_2,L}(x),\dots,p_{F_m,L}(x)$ are linearly independent in $\mathbb{R}^n[x]$ even together with $\mathbf{1}$. In particular, $m\le n$.
\end{theorem}

\begin{proof}
We claim that if $F_1,F_2,\dots,F_m$ are listed in a non-increasing order according to the sizes of the sets, then the polynomials $p_{F_1,L}(x),p_{F_2,L}(x),\dots,p_{F_m,L}(x)$ and the characteristic vectors $v_{F_1},v_{F_2},\dots, v_{F_m}$ satisfy the conditions of Lemma \ref{indep_poly}. Indeed, for any $G\subseteq [n]$ we have $p_{F,k}(G)=\prod_{h\in L}(|F|-|F\cap G|-h)=\prod_{h\in L}(|F\setminus G|-h)$. Therefore $p_{F,L}(v_{F})\neq 0$ holds for any $F\subseteq [n]$, while if $|F_j|\le |F_i|$, then the $L$-close Sperner property ensures $|F_i\setminus F_j|\in L$ and thus $p_{F_j,L}(v_{F_i})=0$.

To prove the moreover part, let $L=\{s\}$, $\cF=\{F_1,F_2,\dots,F_m\}$ and let us suppose towards a contradiction that $\mathbf{1}=\sum_{i=1}^mc_{F_i}p_{F_i,L}(x)$ holds for some reals $c_{F_i}$. We claim that if $|F_i|=|F_j|$, then $c_{F_i}=c_{F_j}$ holds and all coefficients are negative. Observe that for any $F\in\cF$ using the $L$-close Sperner property we have 
\begin{equation}\label{eq}
    1=c_Fp_{F,L}(v_F)+\sum_{\substack{F'\in\cF \\ |F'|>|F|}}c_{F'}p_{F',L}(v_F),
\end{equation}
and $p_{F,L}(v_F)=-s$ for all $F$.
In particular, if $F$ is of maximum size in $\cF$, then $c_F=-\frac{1}{s}$ holds. Let $m_j$ denote $|\{F\in \cF:|F|=j\}|$ and $c_j$ denote the value of $c_F$ for all $F\in \cF$ of size $j$ - once this is proved. By the above, if $j^*$ is the maximum size among sets in $\cF$, then $c_{j^*}$ exists. Suppose that for some $i$ we have proved the existence of $c_j$ for all $j$ with $i<j\le j^*$. If there is no set in $\cF$ of size $i$, there is nothing to prove. If $|F|=i$, then using (\ref{eq}) and the fact $p_{F',L}(v_F)=|F'|-|F|+s-s=|F'|-|F|$ provided $|F'|\ge |F|$, we obtain
\begin{equation}\label{eq2}
    1=c_Fp_{F,L}(v_F)+\sum_{\substack{F'\in\cF \\ |F'|>|F|}}c_{F'}p_{F',L}(v_F)=-sc_F+\sum_{j>i}c_jm_j(j-i).
\end{equation}
This shows that $c_F$ does not depend on $F$ only on $|F|$ as claimed. Moreover, as $s$, $m_j$, $j-i$ are all non-negative and, by induction, all $c_j$ are negative, then in order to satisfy (\ref{eq2}), we must have that $c_i$ is negative as well. So we proved that all $c_j$'s are negative. But this contradicts $\mathbf{1}=\sum_{i=1}^mc_{F_i}p_{F_i,L}(x)$, as on the right hand side all coefficients of the variables are positive, so they cannot cancel. (If there are variables. This is where the condition $\{F_1,F_2\dots,F_m\}\neq \{\emptyset\}$ is used.)
\end{proof}

Using the original "push-to-the-middle" argument of Sperner, it is not hard to prove that for any $k$-close Sperner system $\cF\subseteq 2^{[n]}$, there exists another one $\cF'\subseteq 2^{[n]}$ with $|\cF|=|\cF'|$ and $\cF'$ containing sets of size between $k$ and $n-k$. Is it true that for such set systems we have $\langle p_{F,[k]}: ~F\in \cF'\rangle \cap M_{k-1}=\{\mathbf{0}\}$? This would imply $ex_{sd}(n,[k])=\binom{n}{k}$. 

\bigskip

Let us now turn to the proof of Theorem \ref{induction}. 

\begin{proof}[Proof of Theorem \ref{induction}]
The lower bound is given by the special case $t=1$ of the construction given above Theorem \ref{induction}. It remains to prove the upper bound.

We will prove that a $\{0,1\}$-sd system $\cF\subseteq 2^{[n]}$ is of size at most $\binom{n}{2}+2n-1$ by induction on $n$. Since $\binom{3}{2}+2\cdot 3-1=2^3$, the statement is trivially true for $n=3$. Now assume that $n\ge 4$ and we have already proved the statement for $n-1$.

Consider the uniform systems $\cF_i=\{F\in \cF:|F|=i\}$ that are 1-close Sperner. We will define a representative set $C_i$ for all nonempty levels. If $|\cF_i|\ge 3$, it is an exercise for the reader (see Lemma 19 in \cite{BGMb}) to see that there exists a set $C_i$ either with $|C_i|=i-1$ and $C_i \subseteq \cap_{F\in \cF_i}F$ or with $|C_i|=i+1$ and $\cup_{F \in \cF_i}F\subseteq C_i$. In the former case we say that $\cF_i$ is of type $\vee$, in the latter case we say that $\cF_i$ is of type $\wedge$. If $|\cF_i|=2$, then we select one of the two sets to be $C_i$. If $|\cF_i|=1$, then $C_i$ is the only set in $\cF_i$. Finally, if $\cF_i=\emptyset$, then $C_i$ is undefined.

\begin{claim}
If $i<j$ and $|\cF_i|, |\cF_j|>0$ then $|C_i\backslash C_j|\le 1$.
\end{claim}

\begin{proof}
Assume that there are two different elements $a,b$ such that $a,b\in C_i$ but $a,b\not\in C_j$. It follows from the definition of the representative sets, that there are sets $F_i\in\cF_i$ and $F_j\in\cF_j$ such that $a,b\in F_i$ and $a,b\not\in F_j$. (This is trivial for levels with one or two sets. If there are 3 or more sets then at most two of them can be wrong.)
\end{proof}

Let $C_{p_1}, C_{p_2}, \dots C_{p_t}$ ($p_1<\dots <p_t$) denote the representative sets of the nonempty levels among $\cF_1, \cF_2, \dots \cF_{n-1}$. Since 
$$\left|\bigcup_{i=1}^{t-1} C_{p_i}\backslash C_{p_{i+1}}\right|\le \sum_{i=1}^{t-1} |C_{p_i}\backslash C_{p_{i+1}}|\le t-1\le n-2,$$
there will be an element $x\in[n]$ such that $x\not\in C_{p_i}\backslash C_{p_{i+1}}$ for any $p_i$. This implies that there are no nonempty levels $\cF_i$ and $\cF_j$ such that $i<j$, $x\in C_i$ but $x\not\in C_j$. Rearranging the names of the elements, we may assume that $x=n$.

Now we define two families in $2^{[n-1]}$, let $$\cG=\{F\backslash\{n\}~|~F\in\cF\},~~~\cH=\{H\in 2^{[n-1]}~|~H,H\cup\{n\}\in\cF\}.$$

Note that $|\cF|=|\cG|+|\cH|$. Since $\cG$ is a $\{0,1\}$-sd system in $2^{[n-1]}$, we get an upper bound on its size by induction. We will examine $\cH$ to bound its size as well.

\begin{claim}\label{issubset}
If $A,B\in \cH$ and $|A|<|B|$ then $A\subset B$.
\end{claim}

\begin{proof}
By the definition of $\cH$, we get that $A\cup \{n\}\in\cF$ and $n\not\in B$.
Since $\cF$ is a $\{0,1\}$-sd system, $1\ge |(A\cup \{n\})\backslash B|=|A\backslash B|+1$. Therefore we have $|A\backslash B|=0$ or equivalently $A\subset B$.
\end{proof}

\begin{claim}\label{atmostone}
There is at most one level in $\cH$ with two or more sets in it.
\end{claim}

\begin{proof}
Assume that there are two sets of size $i$ and two sets of size $j$ ($i<j$) in $\cH$. Then in $\cF$ there are two sets of size $i+1$ containing $n$ and two sets of size $j$ that do not contain $n$. From the definition of the representative sets follows that $n\in C_{i+1}$ but $n\not\in C_j$. This is an outright contradiction if $i+1=j$. If $i+1<j$, it contradicts the special property of the element $n$ established earlier.
\end{proof}

\begin{claim}
$|\cH|\le n+1$.
\end{claim}

\begin{proof}
Let $\cH_i=\{H\in \cH:|H|=i\}$ for all $i=0,1,\dots, n-1$. If there is no $i$ such that $|\cH_i|>1$, then $|\cH|\le n$. Assume that $|\cH_t|=k>1$. By Claim \ref{atmostone}, this is the only level with more than one set. If the level $\cH_t$ is of type $\vee$, then the union of its sets is of size $t+k-1$. Claim \ref{issubset} implies that all sets $H\in\cH$, $|H|>t$ must contain this union, therefore the levels $\cH_{t+1}$, $\cH_{t+2}, \dots,$ $\cH_{t+k-2}$ are all empty. If $\cH_t$ is of type $\wedge$, then the intersection of its sets is of size $t-k+1$. Claim \ref{issubset} implies that all sets $H\in\cH$, $|H|<t$ must be subsets of this intersection, therefore the levels $\cH_{t-k+2}$, $\cH_{t-k+3}, \dots,$ $\cH_{t-1}$ are all empty. In either case we get that $|\cH|\le k + (k-2)\cdot 0+ (n-k+1)\cdot 1=n+1$.
\end{proof}

Now we can complete the proof of the theorem:
$$|\cF|=|\cG|+|\cH|\le \binom{n-1}{2}+2(n-1)-1+n+1=\binom{n}{2}+2n-1.$$
\end{proof}

Let us make two final
remarks.
\begin{itemize}
    \item 
    Observe that for the set $L_\ell=\{\ell+1,\ell+2,\dots,n\}$ a system $\cF\subseteq 2^{[n]}$ is $L_\ell$-close Sperner if and only if for every $\ell$-subset $Y$ of $[n]$, the trace $\cF_{[n]\setminus Y}=\{F\setminus Y:F\in \cF\}$ is Sperner. Set systems with this property are called $(n-\ell)$-trace Sperner and results on the maximum size of such systems can be found in Section 4 of \cite{P}.
    \item
    A natural generalization arises in $Q^n=\{0,1,\dots,q-1\}^n$. One can partially order $Q^n$ by $a\le b$ if and only if $a_i\le b_i$ for all $i=1,2,\dots,n$. We say that $A\subseteq \{0,1,\dots,q-1\}^n$ is $L$-close Sperner for some subset $L\subseteq [n]$ if for any distinct $a,b\in A$ we have $sd(a,b):=\min\{|\{i:a_i<b_i\}|,|\{i:a_i>b_i\}|\}\in L$. One can ask for the largest number of  points in an $L$-close Sperner set $A\subseteq Q^n$. Here is a construction for $\{1\}$-close Sperner set: for $2\le i\le n$, $1\le h\le q-1$ let $(v_{i,h})_i=h$, $(v_{i,h})_1=q-h+1$ and $(v_{i,h})_j=0$ if $j\neq i$. Then it is easy to verify that $\{v_{i,h}:2\le i\le n, 1\le h\le q-1\}$ is $\{1\}$-close Sperner of size $(q-1)(n-1)$.
    
    
    An easy upper bound on the most number of points in $Q^n$ that form an $\{1\}$-close Sperner system is $O_q(n^{q-1})$. To see this, for any $a=\{a_1, a_2,\dots a_n\}\in Q^n$ let us define the set $F_a\subseteq [(q-1)n]$ as follows.
    $$F_a:=\bigcup_{i=1}^n \bigcup_{j=1}^{a_i} \{(q-1)(i-1)+j\}$$
    If $A\subseteq Q^n$ is $\{1\}$-close Sperner, then $A'=\{F_a~|~a\in A\}\subset 2^{[(q-1)n]}$ will be $\{1,2,\dots q-1\}$-close Sperner. Theorem \ref{main} implies 
    $$|A|=|A'|\le\sum_{h=0}^{q-1}\binom{(q-1)n}{h}=O_q(n^{q-1}).$$
    We conjecture that for any $q$ there exists a constant $C_q$ such that the maximum number of points in $Q^n$ that form a $\{1\}$-close Sperner system is at most $C_qn$.
    \end{itemize}

\subsubsection*{Acknowledgement}
The research of Nagy was supported by the National Research, Development and Innovation Office - NKFIH under the grants FK 132060 and K 132696.

The research of Patk\'os was supported partially by the grant of Russian Government N 075-15-2019-1926 and by the National Research, Development and Innovation Office - NKFIH under the grants FK 132060 and SNN 129364.

\end{document}